\documentclass[a4paper]{article}

\usepackage[english]{babel}
\usepackage[utf8x]{inputenc}
\usepackage[T1]{fontenc}

\usepackage[a4paper,top=3cm,bottom=2cm,left=3cm,right=3cm,marginparwidth=1.75cm]{geometry}

\usepackage{amsmath}
\usepackage{amsthm}
\usepackage{amssymb}
\usepackage{cancel}
\usepackage{graphicx}
\usepackage[colorinlistoftodos]{todonotes}
\usepackage[colorlinks=true, allcolors=blue]{hyperref}
\newtheorem{theorem}{Theorem}[section]
\newtheorem{lemma}[theorem]{Lemma}
\newtheorem{remark}[theorem]{Remark}
\newtheorem*{conjecture}{Conjecture}
\newtheorem*{thm}{Theorem}
\def\CC{\mathbb{C}}

\def\II{\mathcal{I}}

\title{Quantum version of Wielandt's Inequality revisited}
\author{Mateusz Micha\l{}ek\\
\quad\\
Max Planck Institute\\ Mathematics in the Sciences\\ Leipzig, Germany\\
\texttt{michalek@mis.mpg.de}\\
and\\
Polish Academy of Sciences\\ 
Institute of Mathematics\\ 
Warsaw, Poland
\and
Yaroslav Shitov\\
\quad\\
Izumrudnaya ulitsa 65,\\
kvartira 4\\
Moscow 129346, Russia\\
\texttt{yaroslav-shitov@yandex.ru}
}
\DeclareMathOperator\Endo{End}
\begin{document}
\maketitle

\begin{abstract}
Consider a linear space $L$ of complex $D$-dimensional linear operators, and assume that some power $L^k$ of $L$ is the whole set $\Endo(\mathbb{C}^D)$. Perez-Garcia, Verstraete, Wolf and Cirac conjectured that the sequence $L^1,L^2,\ldots$ stablilizes after $O(D^2)$ terms; we prove that this happens after $O(D^2\log D)$ terms, improving the previously known bound of $O(D^4)$.
\end{abstract}

\section{Introduction}
The main motivation of this article is a conjecture of Perez-Garcia, Verstraete, Wolf and Cirac \cite{PerezGarciaVerstraete}, which can be stated as follows:
\begin{conjecture}
Let $L$ be a linear space of $D\times D$ matrices. If $\dim L^k=D^2$ for some $k$, then it also holds for all $k\geq cD^2$ for some constant $c$ not depending on $D$.
\end{conjecture}
Here by $L^k$ we mean the linear space spanned by all products of $k$ matrices in $L$ - for precise definition see Section \ref{sec:notation}. The state of the art is the main result of \cite{QuantumWielandt}, which instead of the bound $O(D^2)$ provides $O(D^4)$. 

There are several motivations to study this conjecture. The original one is the geometry of uniform Matrix Product States. Indeed, as shown in \cite{PerezGarciaVerstraete} the conjecture has direct consequences on the representations of the $W$-state as matrix product states. Let us cite \cite[Section A]{PerezGarciaVerstraete}:
"\emph{The conjectures, if true, can be used to prove a couple of interesting results, one concerning the
MPS representation of the W-state, and the other concerning
the approximation by MPS of ground
states of gapped hamiltonians.}"

In particular, the conjecture implies that families of tensors described as uniform Matrix Product States may be not closed. 
Further, as explained in \cite{QuantumWielandt} it can be regarded as a quantum analogue of Wielandt's inequality \cite{Wielandt1950}. The result has many other applications including "\emph{dichotomy theorems for the zero--error capacity of quantum channels and for the Matrix Product State (MPS) dimension of ground states of frustration-free Hamiltonians}" and "\emph{new bounds on the required interaction-range of Hamiltonians with unique MPS ground state}" \cite{QuantumWielandt}.

Our main new input is to relate this conjecture to another classical open problem in pure algebra. The question is to bound the $k$ under the assumption that $\dim L^1+L^2+\dots+L^k=D^2$. This is an older open problem posed by Paz \cite{Paz}, who conjectured that the correct optimal bound for $k$ in this setting is $2D-2$. He was able to prove an upper bound of $D^2/3+2/3$, which was later improved to $O\left(D^{1.5}\right)$ by Pappacena~\cite{Pappacena}. The latest best known bound is $O(D \log D)$ by the second author~\cite{Shitov}. The $2D-2$ conjecture is known to hold if $L$ contains a non-derogatory matrix~\cite{GMLS} and for $D\leqslant 5$ (see~\cite{Shitov}). 
This approach leads to our main theorem:

\begin{thm}
Let $L$ be a linear space of $D\times D$ matrices. If $\dim L^k=D^2$ for some $k$, then it also holds for all $k\geq 2D^2\left(6+\log_2(D)\right)$.
\end{thm}
In particular, we confirm that the exponent conjectured in \cite{PerezGarciaVerstraete} is indeed equal to two.

After finishing the articles the authors learned that a related topic has been recently studied by Rahaman in \cite{Rahaman}. Under additional positivity assumptions the author proves an $O(D^2)$ bound for the index of primitivity. This is a related quantity, however it is not associated to a linear subspace of matrices, but rather a (primitive, positive) operator on the space of matrices. Our results remain independent, apart from the fact that the bound we provide is also a bound for the index of primitivity if the operator admits a Kraus decomposition.  
\section*{Acknowledgements}
We would like to thank Khazhgali Kozhasov, Joseph Landsberg, Tim Seynnaeve and Emanuele Ventura for discussions on the topic. MM was supported by Polish National Science Center project 2013/08/A/ST1/00804 affiliated at the University of Warsaw.
\section{Notation}\label{sec:notation}
We fix a complex $D$ dimensional vector space $V\simeq \CC^D$. Let $L=L^1\subset \Endo{V}$ be a subspace of linear endomorphisms of $V$. We fix a basis $A_{1},\dots,A_{\dim L^1}$ of $L$ and regard each $A_i$ as a $D\times D$ matrix. Let $L^j$ be the linear subspace of $\Endo(V)$ spanned by products of (not necessarily distinct) $j$ elements of $L$. In particular, a generator of $L^j$ can be regarded as a \emph{word} $A_{i_1}\cdots A_{i_j}$ of length $j$. More generally for any linear space $S\subset \Endo(V)$ we define:
\begin{enumerate}
\item $S^j\subset \Endo(V)$ as the space generated by products of $j$ elements of $S$,
\item $S^{\leq t}:=\sum_{j=1}^t S^j.$
\end{enumerate}

We say that a matrix $M$ is \emph{zero-square} if $M^2=0$.

\section{Quantum version of Wielandt's Inequality}
Throughout this section we work under the assumption that $\dim L^j=D^2$ for $j$ large enough. We start with a general lemma taken from~\cite{Shitov}.

\begin{lemma}[Claim~13 in~\cite{Shitov}]\label{lem4}
Let $S\subset\CC^{n\times n}$, $P\in\CC^{p\times n}$, $Q\in\CC^{n\times q}$. Let $k$ be the smallest integer such that $PS^kQ\neq0$. Then, for any $A_1,\ldots,A_k\in S$, we have $\operatorname{rank}(PA_1\ldots A_kQ)\leqslant n/k$.\end{lemma}


\begin{lemma}\label{lem5}
Assume that $L^\lambda$ contains a square-zero matrix $H$ of rank $\rho>0$ with $\lambda\rho\leq D(1+\log_2\frac{D}{\rho})$. Then either 

\noindent (1) a square-zero matrix of rank $\rho_1\in\left[1,0.5{\rho}\right]$ is contained in $L^{\lambda_1}$ with 
$\lambda_1\rho_1\leq D(1+\log_2\frac{D}{\rho_1})$, or

\noindent (2) a non-nilpotent matrix of rank at most $\rho$ is contained in $L^{\Lambda}$ with $\Lambda\leqslant\lambda+2D/\rho$.
\end{lemma}

\begin{proof}
We choose a basis such that
$$H=\left(\begin{array}{c|c|c}
O&O&I_\rho\\\hline
O&O&O\\\hline
O&O&O
\end{array}\right)$$
and define $P=(O|O|I_\rho)$ and $Q=(I_\rho|O|O)^\top$. Let $k$ be the smallest integer for which there exist $A_1,\ldots,A_k\in L$ satisfying $PA_1\ldots A_kQ\neq0$ (such an integer exists because $L$ generates the whole matrix ring as a $\CC$-algebra). {Let $A=A_1\ldots A_k$ and $A'=PAQ$ be the bottom left block of $A$.}

\textit{Case 1.} Assume $k\leqslant 2D/\rho$. If $A'$ is not nilpotent, then $HA$ is a non-nilpotent matrix of rank at most $\rho$, which makes the condition (2) valid. Otherwise, $A'$ is a nilpotent of index $\alpha>1$, and then $H_1=(HA)^{\alpha-1} H$ is a square-zero matrix of non-zero rank $\rho_1\leqslant\rho/\alpha$. 
Note that $H_1$ is spanned by words of length at most: 
$$(\alpha-1)(\lambda+k)+\lambda =\alpha\lambda+(\alpha-1)k\leqslant \lambda\rho/\rho_1+2D(\frac{\rho}{\rho_1}-1)/\rho\leqslant\frac{D}{\rho_1}\left(1+\log_2\frac{D}{\rho}+2(1-\frac{\rho_1}{\rho})\right).$$
To prove that condition (1) holds it remains to show that:
$$\log_2\frac{D}{\rho}+2(1-\frac{\rho_1}{\rho})\leqslant \log_2\frac{D}{\rho_1},$$
which is equivalent to:
$$2+\log_2\frac{\rho_1}{\rho}\leqslant 2\frac{\rho_1}{\rho}.$$
One can easily verify this inequality, as $0\leqslant \frac{\rho_1}{\rho}\leqslant\frac{1}{2}$.

\textit{Case 2.} Assume $k\geqslant 2D/\rho$. {Note} that $HAH$ has $A'$ at the upper right block and zeros everywhere else. Lemma~\ref{lem4} shows that the rank of $HAH$ is $\rho_1\leqslant D/k\leqslant 0.5\rho$. {Further,} $HAH$ is spanned by words of length at most $$2\lambda+k\leqslant \lambda\rho/\rho_1+D/\rho_1\leqslant \frac{D}{\rho_1}\left(2+\log_2\frac{D}{\rho}\right)\leqslant \frac{D}{\rho_1}\left(1+\log_2\frac{D}{\rho_1}\right) .$$ {Hence, condition (1) holds.}
\end{proof}

\begin{lemma}\label{lem:nonnil}
There exists $R>0$ and
$$\Lambda\leqslant\frac{D}{R}\left(3+\log_2\frac{D}{R}\right)$$
such that $L^\Lambda$ contains a non-nilpotent matrix of rank $R$.
\end{lemma}

\begin{proof}
If $L$ contains a non-nilpotent matrix, then we are done. Otherwise, there is a matrix $A\in L$ of nilpotency index $\lambda_0+1>1$. {The} matrix $A^{\lambda_0}$ is square-zero, belongs to $L^{\lambda_0}$, and has rank $\rho_0\in[1,D/(\lambda_0+1)]$. Now we repeatedly apply Lemma~\ref{lem5} until we end up under the condition (2) of it; we obtain a sequence $(\lambda_0,\rho_0),\ldots,(\lambda_\tau,\rho_\tau)$. 
We write $R=\rho_\tau$ and assume that we fall into case (2) of Lemma~\ref{lem5} after applying it to $(\lambda_{\tau-1},\rho_{\tau-1})$. 
As $\lambda_{\tau-1}\rho_{\tau-1}\leqslant D(1+\log_2\frac{D}{\rho_{\tau-1}})$ we get:
$$\lambda_\tau\leqslant
\lambda_{\tau-1}+\frac{2D}{\rho_{\tau-1}}\leqslant \frac{D(3+\log_2\frac{D}{\rho_{\tau-1}})}{\rho_{\tau-1}}\leqslant \frac{D}{R}(3+\log_2\frac{D}{R}).$$

\end{proof}

Our aim is to bound from above the smallest $j$ for which $\dim L^j=D^2$. From now on we set:
$$\II :=\min\{j:\dim L^j=D^2\}.$$
The following Lemma is based on the techniques presented in \cite[Section 3]{QuantumWielandt}. We include a complete proof for the sake of completeness.
\begin{lemma}\label{lem:ifnonnil}
Suppose we have a non-nilpotent matrix $B\in L^\Lambda$ of rank $R$. Then $\II\leq \Lambda(R+1)D$.
\end{lemma}
\begin{proof}
{\bf Step 0:}
After rescaling $B$, we may assume there exists an eigenvector $v$ with $Bv=v$. By passing from the sequence $L^j$ to the subsequence $L^{j\cdot \Lambda}$ we may assume that $B\in L^1$ and we want to prove that $\dim L^{(R+1)D}=D^2$.

{\bf Step 1:}
Consider the sequence of vector subspaces of $\CC^D$ defined by:
$$M_j:=(L^1+\dots+L^j)v.$$
Clearly, $M_1\subset M_2\subset \dots$. Further, if $M_{j}=M_{j+1}$, then $M_{j}=M_{j+k}$ for any $k$. Indeed, the former equality is equivalent to $L_{j+1}v\subset (L^1+\dots+L^j)v$. In such a case, by induction on $k$ we have:
$$L^{j+1+k}v=\bigoplus_{A\in L^k}AL^{j+1}v\subset \bigoplus_{A\in L^k,1\leq i\leq j}AL^iv\subset M_{j+k}\subset M_j.$$
Hence, $\dim M_j<\dim M_{j+1}$ unless $M_j=M_{j+1}=\dots$. As the sequence $L^i$ is spanning, we must have $M_j=\CC^D$ for large $j$.
We conclude that $M_k=\CC^D$ for $k\geq D$, by dimension count. It follows that for any $w\in \CC^D$ 
there exist such elements $W_i\in L^i$ that:
$$w=\sum_{i=1}^D W_i v.$$
However, then we also have $w=\sum_{i=1}^D W_iB^{D-i}v$ and $W_iB^{D-i}\in L^D$. We have proved that 
$L^Dv=\CC^D$. 

{\bf Step 2:}
We fix a basis, starting from $v$, in which $B$ is in Jordan normal form. We assume that first $s$ eigenvalues of $B$ are nonzero. Clearly $s\leq R$. Let $P$ be the projector onto the vector space $V'\subset V$ spanned by first $s$ basis vectors. We consider the following spaces of matrices $\tilde M_j:=PL^j$. We claim that $\dim \tilde M_j<\dim\tilde M_{j+1}$, unless $\dim \tilde M_j=sD$, i.e.~it is maximal possible. Indeed, let $W_1,\dots,W_q$ be a basis of $\tilde M_j$. These are linearly independent operators from $V$ to $V'$. We set $\tilde W_i:=BW_i$. As $PB=BP$ we have $\tilde W_i\in M_{j+1}$. Further, as the restriction of $B$ to $V'$ is invertible, we see that $\tilde W_i$ are linearly independent. We see that $\dim \tilde M_j\leq\dim\tilde M_{j+1}$. If equality holds, then $\tilde W_i$ span $\tilde M_{j+1}$. 
In this situation we 
have 
$\tilde M_{j+k+1}=\tilde M_{j+1}L^k=B\tilde M_j L^k=B\tilde M_{j+k}$. In particular, $\dim M_{j+k}$ is constant for $k\geq 0$. As the sequence $L_i$ is spanning this can happen only if $\dim \tilde M_{j}=sD$. By dimension count, it follows that $\dim \tilde M_{sD}=sD$. 

We have $B^{R-s}P=B^{R-s}$, as the nilpotent part of the Jordan decomposition gets annihilated. Hence, $B^{R-s}\tilde M_{sD}=B^{R-s}L_{sD}$. As $B^{R-s}$ restricted to $V'$ is an isomorphism we see that $B^{R-s}L_{sD}\subset L_{sD+R-s}$ contains all linear maps from $V$ to $V'$. In particular, for 
any $w\in V$ there exists such $M\in L_{RD}$ of rank one that $Mw=v$. 

{\bf Step 3:}
We prove that $\dim L^{(R+1)D}=D^2$, by showing that all rank one matrices belong to $L^{(R+1)D}$. Fix arbitrary two vectors $v_1,v_2\in V$. We construct a rank one matrix in $L^{(R+1)D}$ that sends $v_1$ to $v_2$. By Step 1 there exists such a matrix $M_1\in L^D$ that $M_1v=v_2$. By Step 2 there exists such a matrix $M_2\in L_{RD}$ of rank one that $M_2v_1=v$. Clearly $M_1M_2\in L_{(R+1)D}$ is of rank one and $M_1M_2v_1=v_2$, which finishes the proof of the Lemma.
\end{proof}

\begin{theorem}
We have $\II\leq 2D^2\left(6+\log_2 D\right)$, i.e.~$\dim L^k=D^2$ for some $k$ if and only if $\dim L^{\lfloor 2D^2\left(6+\log_2 D\right)\rfloor}=D^2$.
\end{theorem}
\begin{proof}
By Lemma \ref{lem:nonnil} we know there exists a rank $R$ non-nilpotent matrix $A\in L^\Lambda$ with 
$$\Lambda\leqslant\frac{D}{R}\left(3+\log_2\frac{D}{R}\right).$$
Applying this to Lemma \ref{lem:ifnonnil} we obtain:
$$\II\leqslant \left(\frac{D}{R}\left(3+\log_2\frac{D}{R}\right)\right)(R+1)D.$$
As $1\leq R\leq D$ the above value is maximized for $R=1$ which gives the result. 
\end{proof}

\begin{remark}
One could consider a 'dual' question:

Suppose $L^j=0$ for some $j$, what are the bounds on $j$?

This is much easier, as in fact $L^j=0$ if and only if $L^D=0$. Indeed, if $L^{j_0}=0$ for some $j_0$ we know that $\sum_{j=1}^{j_0} L^{j} $ is an algebra of nilpotent matrices. In particular, it is a Lie algebra consisting of nilpotent matrices. Thus by Engel's theorem, all matrices in the algebra can be simultaneously brought into upper-diagonal form. Hence, $L^D=0$. 

Clearly, $D$ is optimal, as demonstrated by an example when $L^1$ consists of all (nilpotent) strictly upper-diagonal matrices.
\end{remark}
\begin{remark}
As one can see, the most problematic case is when $L^1$ contains only nilpotent matrices. Of course, still it is possible that $L^j=\Endo(V)$ for some $V$ - examples can be found e.g.~in \cite{MATHES1991215}.
\end{remark}
\begin{remark}
We point out that even if $L^j=\Endo(V)$ for some $j$ it is not true that the sequence $\dim L^i$ has to be weakly monotonic. An example can be found in \cite{vsidak1964povctu}.
\end{remark}





\bibliographystyle{alpha}
\bibliography{PEPSbib}

\end{document}